\documentclass[sn-mathphys-num]{sn-jnl}

\usepackage{graphicx}%
\usepackage{multirow}%
\usepackage{amsmath,amssymb,amsfonts}%
\usepackage{amsthm}%
\usepackage{mathrsfs}%
\usepackage[title]{appendix}%
\usepackage{xcolor}%
\usepackage{textcomp}%
\usepackage{manyfoot}%
\usepackage{booktabs}%
\usepackage{algorithm}%
\usepackage{algorithmicx}%
\usepackage{algpseudocode}%
\usepackage{listings}%

\theoremstyle{thmstyleone}
\newtheorem{theorem}{Theorem}[section]

\newtheorem{conjecture}[theorem]{Conjecture}
\newtheorem{corollary}[theorem]{Corollary}

\theoremstyle{thmstyletwo}

\theoremstyle{thmstylethree}

\begin{document}
\title[A note on the exact formulas for certain $2$-color partitions]{A note on the exact formulas for certain $2$-color partitions}

\author*[1]{\fnm{Russelle} \sur{Guadalupe}}\email{rguadalupe@math.upd.edu.ph}

\affil*[1]{\orgdiv{Institute of Mathematics}, \orgname{University of the Philippines}, \orgaddress{\street{Diliman}, \city{Quezon City}, \postcode{1101}, \country{Philippines}}}

\abstract{Let $p\leq 23$ be a prime and $a_p(n)$ counts the number of partitions of $n$ where parts that are multiple of $p$ come up with $2$ colors. Using a result of Sussman, we derive the exact formula for $a_p(n)$ and obtain an asymptotic formula for $\log a_p(n)$. Our results partially extend the work of Mauth, who proved the asymptotic formula for $\log a_2(n)$ conjectured by Banerjee et al.}



\maketitle

\section{Introduction}\label{sec1}

Throughout this paper, we denote $(a;q)_\infty = \prod_{n\geq 0}(1-aq^n)$ for $a\in\mathbb{C}$. Recall that a partition of a positive integer $n$ is a nonincreasing finite sequence of positive integers, known as its parts, whose sum is $n$. We define $p(n)$ as the number of partitions of $n$, which can be seen as the coefficients of its generating function given by
\begin{align*}
\dfrac{1}{(q;q)_\infty} = \sum_{n=0}^\infty p(n)q^n.
\end{align*}  
Hardy and Ramanujan \cite{hadram} proved the following asymptotic formula for $p(n)$ given by
\begin{align*}
p(n) \sim \dfrac{1}{4n\sqrt{3}}\exp\left(\pi\sqrt{\dfrac{2n}{3}}\right), \quad n\rightarrow\infty,
\end{align*}
using the celebrated Circle Method. Rademacher \cite{rade} refined the Hardy-Ramanujan Circle Method and derived the exact formula
\begin{align*}
p(n) = \dfrac{1}{\pi\sqrt{2}}\sum_{k=1}^\infty A_k(n)\sqrt{k}\dfrac{d}{dn}\left(\dfrac{\sinh(\frac{\pi}{k}\sqrt{\frac{2}{3}(n-\frac{1}{24})})}{\sqrt{n-\frac{1}{24}}}\right)
\end{align*}
where 
\begin{align*}
A_k(n) := \sum_{\substack{0\leq h < k\\ \gcd(h,k)=1}} \exp\left[\pi i \left(s(h,k)-\dfrac{2nh}{k}\right)\right]
\end{align*}
and 
\begin{align*}
s(h,k) := \sum_{j=1}^{k-1}\dfrac{j}{k}\left(\left\{\dfrac{hj}{k}\right\}-\dfrac{1}{2}\right)
\end{align*}
is the Dedekind sum, where $\{t\}$ denotes the fractional part of $t$. Recently, Banerjee et al. \cite{baner1} proved the following refined asymptotic formula for $p(n)$ given by
\begin{align*}
\log p(n)\sim \pi\sqrt{\dfrac{2n}{3}}-\log n-\log 4\sqrt{3}-\dfrac{0.44\cdots}{\sqrt{n}}, \quad n\rightarrow\infty,
\end{align*}
from a family of inequalities for $p(n)$, which was used to prove a refine of the inequality of DeSalvo and Pak \cite{salvop}, and Chen, Wang and Xie \cite{chenwx} that reads
\begin{align*}
\left(1+\dfrac{\pi}{24n^{3/2}}-\dfrac{1}{n^2}\right)p(n-1)p(n+1) < p(n)^2 < \left(1+\dfrac{\pi}{24n^{3/2}}\right)p(n-1)p(n+1)
\end{align*}
for $n\geq 120$.

Let $a_k(n)$ be the number of partitions of $n$ where parts that are multiple of $k$ come up with $2$ colors. For $k=1$, we have $a_1(n)= p_2(n)$ where $p_2(n)$ is the total number of partitions of $n$ where parts come up with $2$ colors and for $k=2$, $a_2(n)$ also counts the number of cubic partitions (see \cite{bkim}). The generating function for $a_k(n)$ is given by \cite{ahmed}
\begin{align*}
\dfrac{1}{(q;q)_\infty(q^k;q^k)_\infty} = \sum_{n=0}^\infty a_k(n)q^n.
\end{align*}  
Kotesovec \cite{kotes} found the following asymptotic formula for $a_2(n)$ given by 
\begin{align*}
a_2(n) \sim \dfrac{1}{8n^{5/4}}\exp(\pi\sqrt{n}), \quad n\rightarrow\infty.
\end{align*}
Banerjee et al. \cite{baner1} conjectured that $a_2(n)$ satisfies the following asymptotic formula
\begin{align}
\log a_2(n)\sim \pi\sqrt{n}-\dfrac{5}{4}\log n-\log 8-\dfrac{0.79\cdots}{\sqrt{n}}, \quad n\rightarrow\infty. \label{eq1}
\end{align}
Recently, Mauth \cite{mauth} proved (\ref{eq1}) by finding the exact formula of Rademacher type for $a_2(n)$ using a result of Zuckerman \cite{zucker}.

In this paper, we use a result of Sussman \cite{suss} to derive Rademacher-type formulas for $a_p(n)$ when $p\leq 23$ is a prime, and deduce their asymptotic formulas in the spirit of (\ref{eq1}), which can be considered as a partial extension of Mauth's work. We give our main results as follows. 

\begin{theorem}\label{thm11}
Let $p\leq 23$ be a prime. Then for $n\geq 1$ we have
\begin{align*}
	a_p(n) &= 2\pi\sqrt{p}\left(\dfrac{1+p^{-1}}{24n-p-1}\right)\sum_{j=1}^{p-1}\sum_{\substack{m=1\\ m\equiv j\bmod{p}}}^\infty I_2\left(\dfrac{\pi}{6m}\sqrt{(1+p^{-1})(24n-p-1)}\right)\dfrac{b_m(n)}{m}\\
	&+ 2\pi\left(\dfrac{1+p}{24n-p-1}\right)\sum_{\substack{m=1\\ p\mid m}}^\infty I_2\left(\dfrac{\pi}{6m}\sqrt{(1+p)(24n-p-1)}\right)\dfrac{b_m(n)}{m},
\end{align*}
where 
\begin{align*}
	b_k(n) := \sum_{\substack{0\leq h < k\\ \gcd(h,k)=1}} \exp\left[-\dfrac{2\pi nhi}{k}+\pi s(h,k)i+\pi s\left(\dfrac{ph}{\gcd(p,k)},\dfrac{k}{\gcd(p,k)}\right)i\right]
\end{align*}
and $I_2(s)$ is the second modified Bessel function of the first kind (see Section \ref{sec2}).
\end{theorem}

\begin{corollary}\label{cor12}
For $p\leq 23$ prime, we have 
\begin{align*}
	\log a_p(n) \sim \pi\sqrt{\dfrac{2n(1+p^{-1})}{3}}-\dfrac{5}{4}\log n + \log \dfrac{2\sqrt{3p}(1+p^{-1})^{3/4}}{24^{5/4}}-\dfrac{c_p}{24\sqrt{6n}},\quad n\rightarrow\infty
\end{align*}
where 
\begin{align*}
	c_p := \dfrac{135}{\pi\sqrt{1+p^{-1}}}+\pi\sqrt{\dfrac{(1+p)^3}{p}}.
\end{align*}
\end{corollary}

Setting $p=2$ in Corollary \ref{cor12}, we deduce the asymptotic formula due to Mauth \cite{mauth} given by
\begin{align*}
\log a_2(n)\sim \pi\sqrt{n}-\dfrac{5}{4}\log n-\log 8-\left(\dfrac{15}{8\pi}+\dfrac{\pi}{16}\right)\dfrac{1}{\sqrt{n}}, \quad n\rightarrow\infty,
\end{align*}
which immediately yields (\ref{eq1}).

The paper is organized as follows. In Section \ref{sec2}, we state the Sussman's result on the exact formulas for the Fourier coefficients of a class of $\eta$-quotients. In Section \ref{sec3}, we apply this result to prove Theorem \ref{thm11}, and using the asymptotic expansion of $I_2(s)$ due to Banerjee \cite{baner2}, we deduce Corollary \ref{cor12}.

\section{Fourier coefficients of a class of $\eta$-quotients}\label{sec2}

We consider in this section the exact formulas for the Fourier coefficients of the following class of holomorphic functions on the open disk given by
\begin{align*}
G(q) := \prod_{r=1}^R(q^{m_r};q^{m_r})_\infty^{\delta_r} =\sum_{n=0}^\infty g(n)q^n,
\end{align*}
where $\mathbf{m}=(m_1,\ldots,m_R)$ is a sequence of $R$ distinct positive integers and $\mathbf{\delta}=(\delta_1,\ldots,\delta_R)$ is a sequence of $R$ nonzero integers. The functions $G(q)$ can be seen as $\eta$-quotients since $(q;q)_\infty=q^{-1/24}\eta(\tau)$, where Dedekind's eta function, denoted
by $\eta(\tau)$, is defined by the infinite product $\eta(\tau):=q^{1/24}\prod_{n\geq 1}(1-q^n)$ with $q=e^{2\pi i\tau}$and $\tau\in\mathbb{H}:=\{z\in\mathbb{C} : \mbox{Im}(z)>0\}$. For useful properties of the Dedekind's eta function and the definition of $\eta$-quotients, we refer the interested reader to the book \cite{ono}. Given a particular $G(q)$ with $\sum_{r=1}^R \delta_r < 0$, Sussman \cite{suss} gave a Rademacher-type exact formula for $g(n)$, which is a special case of the work of Bringmann and Ono \cite{bringn} on the coefficients of harmonic Maass forms. Sussman's proof follows the original approach of Rademacher \cite{rade} on the Hardy-Ramanujan Circle Method. In the case where $\sum_{r=1}^R \delta_r \geq 0$, Chern \cite{chern} obtained an analogous formula with an error term using the method of O-Y. Chan \cite{oych}.

Before we state Sussman's result, we need some definitions. For $(h, k)\in\mathbb{N}^2$ with $\gcd(h, k) = 1$, we set
\begin{align*}
\Delta_1 &= -\dfrac{1}{2}\sum_{r=1}^R \delta_r, \quad \Delta_2 = \sum_{r=1}^R m_r\delta_r, \\
\Delta_3(k) &= -\sum_{r=1}^R \dfrac{\gcd(m_r,k)^2}{m_r}\delta_r, \quad \Delta_4(k) = \prod_{r=1}^R \left(\dfrac{\gcd(m_r,k)}{m_r}\right)^{\delta_r/2},\quad\text{ and }\\
\widehat{A}_k(n) &= \sum_{\substack{0\leq h < k\\ \gcd(h,k)=1}}\exp\left(-\dfrac{2\pi nhi}{k}-\pi i\sum_{r=1}^R\delta_r s\left(\dfrac{m_rh}{\gcd(m_r,k)},\dfrac{k}{\gcd(m_r,k)}\right)\right),
\end{align*}
where $s(h,k)$ is the Dedekind sum defined in Section \ref{sec1}. We also set $L$ as the least common multiple of $m_1,\ldots,m_R$, and partition the set $\{1,\ldots,L\}$ into two disjoint subsets
\begin{align*}
\mathcal{L}_{>0} &:= \{1\leq l\leq L: \Delta_3(l) > 0\},\\
\mathcal{L}_{\leq 0} &:= \{1\leq l\leq L: \Delta_3(l) \leq 0\}.
\end{align*}

\begin{theorem}[\cite{suss}]\label{thm21}
If $\Delta_1 > 0$ and the inequality
\begin{align}
	\min_{1\leq r\leq R} \dfrac{\gcd(m_r,l)^2}{m_r}\geq \dfrac{\Delta_3(l)}{24}\label{eq2}
\end{align}
holds for $1\leq l \leq L$, then for positive integers $n > -\Delta_2/24$, we have
\begin{align*}
	g(n) = 2\pi\sum_{l\in \mathcal{L}_{>0}}\Delta_4(l)\left(\dfrac{\Delta_3(l)}{24n+\Delta_2}\right)^{(\Delta_1+1)/2}\sum_{\substack{k=1\\ k\equiv l\bmod{L}}}^\infty I_{\Delta_1+1}\left(\dfrac{\pi}{6k}\sqrt{\Delta_3(l)(24n+\Delta_2)}\right)\dfrac{\widehat{A}_k(n)}{k},
\end{align*}
where 
\begin{align*}
	I_\nu(s) := \sum_{m=0}^\infty \dfrac{(\frac{s}{2})^{\nu+2m}}{m!\Gamma(\nu+m+1)}
\end{align*}
is the $\nu$th modified Bessel function of the first kind, and $\Gamma(s)=\int_0^\infty e^{-t}t^{s-1}\,dt$ is the gamma function.
\end{theorem}

\section{Proofs of Theorem \ref{thm11} and Corollary \ref{cor12}}\label{sec3}

In this section, we use Theorem \ref{thm21} to prove Theorem \ref{thm11}. As a consequence, we obtain Corollary \ref{cor12} using the asymptotic expansion of the modified Bessel function of the first kind due to Banerjee \cite{baner2}.

\begin{proof}[Proof of Theorem \ref{thm11}]
Writing
\begin{align*}
	\dfrac{1}{(q;q)_\infty(q^p;q^p)_\infty} = \sum_{n=0}^\infty a_p(n)q^n,
\end{align*}
we have $\mathbf{m}=(1,p), \mathbf{\delta}=(-1,-1)$ and $L=p$. We compute $\Delta_1=1$ and $\Delta_2=-p-1$. Recalling that $p$ is a prime, we have that for $l\in\{1,\ldots, p\}$,
\begin{align*}
	\Delta_3(l) =\begin{cases}
		1+\dfrac{1}{p}, & l\neq p\\
		1+p, & l=p
	\end{cases},\qquad
	\Delta_4(l) =\begin{cases}
		\sqrt{p}, & l\neq p\\
		1, & l=p
	\end{cases},
\end{align*} 
so that $L_{>0}=\{1,\ldots, p\}$. Since $p\leq 23$, we see that condition (\ref{eq2}) holds for $l\in \{1,\ldots, p\}$. Applying Theorem \ref{thm21} with $\widehat{A}_k(n)=b_k(n)$, we obtain
\begin{align*}
	a_p(n) = 2\pi\sum_{l\in \mathcal{L}_{>0}}\Delta_4(l)\left(\dfrac{\Delta_3(l)}{24n-p-1}\right)\sum_{\substack{m=1\\ m\equiv l\bmod{p}}}^\infty I_2\left(\dfrac{\pi}{6m}\sqrt{\Delta_3(l)(24n-p-1)}\right)\dfrac{b_m(n)}{m}
\end{align*}
for $n>(p+1)/24$. Plugging in the values of $\Delta_3(l)$ and $\Delta_4(l)$ yields the desired result.
\end{proof}

We see that the series for $a_p(n)$ in Theorem \ref{thm11} converges rapidly and that the term $m=1$ contributes signficantly to the series. We thus obtain the asymptotic formula
\begin{align}
a_p(n) \sim 2\pi\sqrt{p}\left(\dfrac{1+p^{-1}}{24n-p-1}\right)I_2\left(\dfrac{\pi}{6}\sqrt{(1+p^{-1})(24n-p-1)}\right),\quad n\rightarrow\infty.\label{eq3}
\end{align}
To prove Corollary \ref{cor12}, we use the following asymptotic formula \cite[eq. (2.11)]{baner2}
\begin{align}
I_\nu(s) \sim \dfrac{e^s}{\sqrt{2\pi s}}\sum_{m=0}^\infty \dfrac{(-1)^md_m(\nu)}{x^m}, \quad s\rightarrow\infty,\label{eq4}
\end{align}
where 
\begin{align*}
d_m(\nu) = \dfrac{\binom{\nu-1/2}{m}(\nu+\frac{1}{2})_m}{2^m}
\end{align*}
with 
\begin{align*}
\dbinom{a}{m}&=\begin{cases}
	\dfrac{a(a-1)\cdots(a-m+1)}{m!}, & m\in \mathbb{N}\\
	1, & m=0
\end{cases},\\
(a)_m&=\begin{cases}
	a(a+1)\cdots (a+m-1), & m\in \mathbb{N}\\
	1, & m=0
\end{cases}
\end{align*}
for $a\in\mathbb{R}$.

\begin{proof}[Proof of Corollary \ref{cor12}]
From (\ref{eq4}) we get
\begin{align*}
	I_2(s) = \dfrac{e^s}{\sqrt{2\pi s}}\left(1-\dfrac{15}{8s}+O\left(\dfrac{1}{s^2}\right)\right).
\end{align*}
In view of (\ref{eq3}),  we have
\begin{align}
	a_p(n)\sim 2\sqrt{3p}&\cdot\dfrac{(1+p^{-1})^{3/4}}{(24n-p-1)^{5/4}}\nonumber\\
	&\cdot\exp\left(\dfrac{\pi}{6}\sqrt{(1+p^{-1})(24n-p-1)}\right)\left(1-\dfrac{45}{8\pi\sqrt{6n(1+p^{-1})}}+O\left(\dfrac{1}{n}\right)\right).\label{eq5}
\end{align}
Since 
\begin{align*}
	\dfrac{1}{(24n-p-1)^{5/4}}&=\dfrac{1}{(24n)^{5/4}}\left(1+O\left(\dfrac{1}{n}\right)\right),\\
	\exp\left(\dfrac{\pi}{6}\sqrt{(1+p^{-1})(24n-p-1)}\right) &=\exp\left(\pi\sqrt{\dfrac{2n(1+p^{-1})}{3}}\right)\left(1-\dfrac{p^{-1/2}(1+p)^{3/2}\pi}{24\sqrt{6n}}+O\left(\dfrac{1}{n}\right)\right),
\end{align*}
we infer from (\ref{eq5}) that
\begin{align*}
	a_p(n)\sim  2\sqrt{3p}&\cdot\dfrac{(1+p^{-1})^{3/4}}{(24n)^{5/4}}\left(1+O\left(\dfrac{1}{n}\right)\right)\cdot\exp\left(\pi\sqrt{\dfrac{2n(1+p^{-1})}{3}}\right)\\
	&\cdot\left(1-\dfrac{p^{-1/2}(1+p)^{3/2}\pi}{24\sqrt{6n}}+O\left(\dfrac{1}{n}\right)\right)\left(1-\dfrac{45}{8\pi\sqrt{6n(1+p^{-1})}}+O\left(\dfrac{1}{n}\right)\right),
\end{align*}
and taking logarithms yields the desired conclusion.
\end{proof}
We remark that when $p > 23$ is a prime, condition (\ref{eq2}) of Theorem \ref{thm21} fails. Nonetheless, by numerical experiments via \textit{Mathematica}, we propose the following conjecture. 

\begin{conjecture}
Corollary \ref{cor12} also holds for primes $p > 23$.
\end{conjecture}

\bibliography{twocolor}
\end{document}